\documentclass{amsart}

\usepackage{graphicx}

\newtheorem{Theorem}{Theorem}
\newtheorem{Claim}{Claim}
\newtheorem{Corollary}[Theorem]{Corollary}
\newtheorem{Lemma}{Lemma}

\theoremstyle{remark}
\newtheorem{Remark}{Remark}

\begin{document}

\title{Knots with arbitrarily high distance bridge decompositions}

\author{Kazuhiro Ichihara}
\address{Department of Mathematics, 
College of Humanities and Sciences, Nihon University,
3-25-40 Sakurajosui, Setagaya-ku, Tokyo 156-8550, Japan}
\email{ichihara@math.chs.nihon-u.ac.jp}
\thanks{The first author is partially supported by
Grant-in-Aid for Young Scientists (B), No.~23740061, 
Ministry of Education, Culture, Sports, Science and Technology, Japan.}

\author{Toshio Saito}
\address{Department of Mathematics,
Joetsu University of Education,
1 Yamayashiki, Joetsu 943-8512, Japan}
\email{toshio@juen.ac.jp }
\thanks{The second author is partially supported by
Grant-in-Aid for Young Scientists (B), No.~24740041, 
Ministry of Education, Culture, Sports, Science and Technology, Japan.}

\begin{abstract}
We show that 
for any given closed orientable 3-manifold $M$ 
with a Heegaard surface of genus $g$, 
any positive integers $b$ and $n$, 
there exists a knot $K$ in $M$ 
which admits a $(g,b)$-bridge splitting 
of distance greater than $n$ 
with respect to the Heegaard surface 
except for $(g,b)=(0,1), (0,2)$. 
\end{abstract}



\date{\today}

\maketitle

\section{Introduction}

A \textit{Heegaard splitting} of a closed, orientable 3-manifold $M$ is 
a decomposition of $M$ into two handlebodies, 
which was originally introduced and studied by Heegaard \cite{Heegaard}. 
About fifty years later, Moise proved in \cite{Moise} that 
every closed orientable 3-manifold can admit a Heegaard splitting. 
It is then shown that certain properties of Heegaard splittings 
reflect topological characteristics of 3-manifolds: 
For example, 
any Heegaard splitting of a reducible 3-manifold is reducible \cite{Haken}, 
any irreducible Heegaard splitting of a non-Haken 3-manifold is 
strongly irreducible \cite{CassonGordon}. 

Motivated by such works, 
Hempel \cite{Hempel} introduced 
an integer-valued invariant of a Heegaard splitting, 
called the \textit{distance}, 
or commonly called the \textit{Hempel distance}. 
This is a natural refinement of strong irreducibility, 
and measures certain complexity of Heegaard splittings. 
A lot of studies have been done about 
the distance of Heegaard splittings. 
See \cite[Section 6]{Minsky} for an aspect of such works. 

Among them, on the existence of high distance splittings, 
there are several known results. 
We here summarize a part of them in the following. 

First, Hempel \cite{Hempel} exhibited that there exist 
Heegaard splittings of closed orientable 3-manifolds 
with distance at least $n$ for arbitrarily large $n$ 
(adapting an idea of Kobayashi \cite{Kobayashi}). 
Later, Evans \cite{Evans} explicitly defined an infinite sequence 
of closed 3-manifolds $\{ M_n \}$ and proved that the distance of 
a Heegaard splitting of $M_n$ is at least $n$ 
by using purely combinatorial techniques. 

About the non-closed compact 3-manifold case, 
in particular, the knot exterior case, 
Minsky, Moriah and Schleimer \cite{MinskyMoriahSchleimer} 
constructed knots in the 3-sphere $S^3$ with the exteriors 
having Heegaard splittings of arbitrarily high distance, in any genus. 
Extending their results, 
Campisi and Rathbun \cite{CampisiRathbun} showed that 
for any closed 3-manifold $M$ 
with a minimal genus Heegaard splitting of genus $g(M)$ and 
any integers $g \ge g(M) + 1$ and $n > 0$, 
there is a knot $K$ in $M$ with its exterior 
admitting a genus $g$ Heegaard splitting of distance greater than $n$.

A natural generalization of the Heegaard splitting 
for a link in a 3-manifold is given by the bridge splitting. 
A \textit{$(g,b)$-bridge splitting}
of a link $L$ in a closed $3$-manifold $M$ 
is a decomposition of $(M,L)$ into 
two pairs of a genus $g$ handlebody and $b$ trivial arcs.  
For bridge splittings, the notion of distance is naturally defined 
as a generalization of the case of Heegaard splittings for closed manifolds. 
See the next section for details. 

The second author \cite{Saito} showed that 
for any closed $3$-manifold admitting a Heegaard splitting of genus one, 
there is a knot in the manifold 
with a $(1,1)$-bridge splitting of arbitrarily high distance. 
Recently, Blair, Tomova and Yoshizawa \cite{BlairTomovaYoshizawa} 
showed that for given integers $b, c, g$, and $n$, 
there exists a manifold $M$ containing a $c$-component link $L$ 
so that $(M,L)$ admits a $(g,b)$-bridge splitting of distance at least $n$.

Extending their results, in this paper, we show the following. 

\begin{Theorem}\label{Thm}
For any given closed orientable 3-manifold $M$ 
with a Heegaard surface of genus $g$, 
any positive integers $b$ and $n$, 
there exists a knot $K$ in $M$ 
which admits a $(g,b)$-bridge splitting 
of distance greater than $n$ 
with respect to the Heegaard surface 
except for $(g,b)=(0,1), (0,2)$. 
\end{Theorem}

\begin{Remark}
If we regard a $(g,0)$-bridge splitting for a knot 
as a Heegaard splitting of the exterior of the knot, 
then the case where $b=0$ corresponds 
to the result given in \cite{CampisiRathbun}. 
\end{Remark}

\begin{Remark}
Note that the case $(g,b)=(0,1)$ have to be excluded, 
for the distance of such a splitting cannot be defined. 
Also, for the case $(g,b)=(0,2)$, 
if we change some definitions to define the distance 
for $(0,2)$-bridge splittings suitably, 
then the same result can be obtained. 
See the remarks in Section 2. 
\end{Remark}

\begin{Remark}
In general, we can also make 
a link of at most $b$ components in stead of the knot 
enjoying the desired properties in the theorem. 
See remarks in Section 3. 
\end{Remark}

As an application, in Section 4, we will study 
the \textit{meridional destabilizing numbers} of knots in $S^3$ 
introduced by the second author in \cite{Saito2}. 
Recall that, for a knot in $S^3$, an invariant is defined by 
the minimal genus of Heegaard surfaces of its exterior minus one, 
called the \textit{tunnel number} of a knot. 
This gives an interesting subject to study in Knot Theory, and 
there are many works to study it. 
By considering bridge splittings of knots simultaneously, 
for example, tunnel number one knots are classified into 3 classes: 
the non-trivial two-bridge knots, i.e., knots admitting $(0,2)$-bridge splittings, 
the $(1,1)$-knots, i.e., knots admitting $(1,1)$-bridge splittings, and 
the other knots, i.e., knots admitting $(2,0)$-bridge splittings. 
A natural generalization of this filtration is given by 
considering the meridional destabilizing number, and 
as a corollary to Theorem 1, we have the following.

\begin{Corollary}\label{cor}
For any integers $t \ge 1$ and $m \ge 0$ with $m \le t+1$, 
there exists a knot $K\subset S^3$ 
of tunnel number $t$ and of meridional destabilizing number $m$.
\end{Corollary}

Precise definition of the meridional destabilizing number of knots and 
a proof of this corollary will be given in the last section.

\section*{Acknowledgments}

The authors would like to thank to 
Michael Yoshizawa, Ayako Ido, and Yeonhee Jang 
for useful discussions and information. 
They also thank Maggy Tomova for 
information about Lemma 1 in Section 4, 
Hideki Miyachi for helpful comments about the proof of Claim 1, 
and Yo'av Rieck for pointing out inaccurate statements in our previous draft and for comments on Remark 6. 
Finally they thank to the anonymous referee for his/her careful readings.

\section{Distance of Heegaard splitting and $(g,b)$-bridge splitting}

In this section, we recall definitions of 
a Heegaard splitting of a 3-manifold, 
a $(g,b)$-bridge splitting of a knot, and their (Hempel) distance. 

Let $M$ be a closed orientable $3$-manifold. 
A closed orientable surface $S$ of genus $g$ embedded in $M$ 
is called a (genus $g$) \textit{Heegaard surface} of $M$ 
if $S$ divides $M$ into two handlebodies $V_1$ and $V_2$. 
The triplet $(V_1,V_2;S)$ is called 
a (genus $g$) \textit{Heegaard splitting} of $M$. 

Let $K$ be a knot in a closed orientable $3$-manifold $M$. 
We say that $K$ admits a \textit{$(g,b)$-bridge splitting} ($b>0$) 
if there is a genus $g$ Heegaard splitting $(V_1,V_2;S)$ of $M$ 
such that $V_i\cap K$ consists of $b$ arcs 
which are mutually boundary parallel for each $i=1,2$. 
Set $\mathcal{V}_i=(V_i,V_i\cap K)$ and $\mathcal{S}=(S,S\cap K)$. 
We also call the triplet $(\mathcal{V}_1,\mathcal{V}_2;\mathcal{S})$ 
a \textit{$(g,b)$-bridge splitting} of $(M,K)$, 
and $\mathcal{S}$ is called a \textit{$(g,b)$-bridge surface}, 
or a \textit{bridge surface} for short. 
We say that $K$ admits a \textit{$(g,0)$-bridge splitting} 
if there is a genus $g$ Heegaard splitting $(V_1,V_2;S)$ of $M$ 
such that $K\subset V_i$ ($i=1$ or $2$), 
say $i=2$, and that $V_2- \mathrm{int}N(K)$ is a compression body. 
Note that a $(g,0)$-bridge splitting of $(M,K)$ 
is also called a \textit{Heegaard splitting} of $(M,K)$, and 
a $(g,0)$-bridge surface of $(M,K)$ is called a \textit{Heegaard surface} of $(M,K)$. 

To define the (Hempel) distance of the splittings above, 
we first prepare the terminology about the curve complex, 
originally introduced in \cite{Harvey81, Harvey88}. 
Let $F$ be a compact orientable surface 
of genus $g$ possibly with non-empty boundary. 
Then the \textit{curve complex} $\mathcal{C}(F)$ of $F$ is 
defined as the simplicial complex whose $k$-simplexes are 
the isotopy classes of $k+1$ collections of mutually non-isotopic 
essential loops (i.e., non-contractible and not boundary-parallel) 
which can be realized disjointly. 
It is then proved in \cite{MasurMinsky} that 
the curve complex is connected 
if $F$ is not sporadic, 
i.e., $F$ has at least 5 (resp. 2) boundary components if $g=0$ (resp. $g=1$). 
For vertices $[x]$ and $[y]$ of $\mathcal{C}(F)$, 
the distance $d([x],[y])$ between $[x]$ and $[y]$ is defined as 
the minimal number of $1$-simplexes in a simplicial path joining $[x]$ to $[y]$. 

For a bridge splitting $(\mathcal{V}_1,\mathcal{V}_2;\mathcal{S})$ of $(M,K)$, 
set $W_i=V_i-\textrm{int}(V_i\cap K)$ $(i=1,2)$ 
and $S_0=S-\textrm{int}N(S\cap K)$. 
For $i=1,2$, the set of vertices  in $\mathcal{C}(S_0)$ 
each of which corresponds to a curve bounding a disk in $W_i$ 
is called the \textit{disk set} $\mathcal{D}(W_i)$. 
Then the \textit{Hempel distance} $d(\mathcal{S})$ is defined as 
$d(\mathcal{D}(W_1), \mathcal{D}(W_2))
=\min \{d([x],[y])\ |\ [x]\in \mathcal{D}(W_1),[y]\in \mathcal{D}(W_2)\}$. 

\begin{Remark}
For the case of $(g,b)=(0,2)$, 
equivalently, for the case of two-bridge knots in $S^3$, 
the corresponding curve complex in the above definition is not connected, 
and so the distance of such a splitting cannot be well-defined. 
However, as remarked in Section 1, 
the same result as Theorem 1 can be obtained for this case 
if we change some definition suitably. 
To do this, we define 
the curve complex of the 4-punctured sphere 
as the simplicial complex whose $k$-simplexes are 
the isotopy classes of $k+1$ collections of mutually non-isotopic 
essential loops (i.e., non-contractible and not boundary-parallel) 
which can be realized with just two points as intersections. 
Then the curve complex can be naturally identified with 
the well-known Farey diagram. 
The distance for a given splitting, 
equivalently a given two-bridge knot, 
can be related to 
the minimal length of the continued fractional expansions 
for the rational number corresponding to the two-bridge knot. 
See \cite{HatcherThurston} for example. 
\end{Remark}

\section{Proof of Theorem~\ref{Thm}}

In this section, we give a proof of Theorem 1. 
We can make use of the strategy 
for the proof of the main theorem in \cite{MinskyMoriahSchleimer}. 
Also see \cite[Section 4]{CampisiRathbun} as a summary. 
Actually their theorems concern only closed 3-manifolds, 
and so, we modify the arguments to fit for the case of manifolds with boundary. 

\subsection{Settings}

We start with preparing our setting. 

Let $M$ be a closed orientable 3-manifold 
and $F$ a Heegaard surface of $M$ of genus $g$. 
That is, $F$ decomposes $M$ into two handlebodies $V_1$ and $V_2$. 
Choose arbitrary integers $n, b \ge 1$. 
We only consider the case $b \ge 3$ if $g=0$ by the assumption. 

As stated in Section 1, 
the second author \cite{Saito} showed that 
for any closed $3$-manifold admitting a Heegaard splitting of genus one, 
there is a knot in the manifold 
with a $(1,1)$-bridge splitting of arbitrarily high distance. 
Thus, in the following, 
we consider the case $b \ge 2$ if $g=1$. 

Let us first take a $2b$-tuple of points 
$\{ x_1 , x_2, \cdots, x_{2b-1} , x_{2b} \}$ on $F$. 
Take embedded $b$ arcs $\alpha_1 , \cdots , \alpha_b$ on $F$ 
disjoint from each other 
such that $\alpha_j$ connects $x_{2j-1}$ and $x_{2j}$ for $j = 1, \cdots, b$.
Pushing the interior of each arc $\alpha_j$ into $V_i$, 
we have a pair of $b$ boundary-parallel arcs in both handlebodies, 
say, $\alpha_{i,1} , \alpha_{i,2} , \cdots , \alpha_{i,b} \subset V_i$ for $i=1,2$. 
Then they compose a $b$-component trivial link in $M$. 

Take a small regular neighborhood $N(F)$ of $F$ in $M$, 
which is topologically $F \times [0,1]$, 
and regard it as lying in $V_2$ with $F=F \times \{ 0 \}$. 
Now we define $W_1 = V_1 - \mathrm{int} N(\alpha_{1,1} \cup  \cdots \cup \alpha_{1,b})$ 
and $W_2 = V_2 - \mathrm{int} N(\alpha_{2,1} \cup  \cdots \cup \alpha_{2,b})$.

\subsection{Key ingredients}

We here prepare two key ingredients used in the proof of Theorem 1. 
One is a pseudo-Anosov homeomorphism $\phi$ on $F_0:=F-\mathrm{int} N(x_1\cup \cdots \cup x_{2b})$, 
and the other is a simple closed curve $c$ on $F_0$, 
both satisfying some particular conditions.

\subsubsection{pseudo-Anosov map}

A \textit{pseudo-Anosov} homeomorphism $\phi : F_0 \to F_0$ 
is defined as a homeomorphism on $F_0$ 
with $\phi|_{\partial}=\mathit{id}$ such that 
there exist a constant $k_{\phi} > 1$ and 
a pair of transverse measured foliations $\lambda^\pm$ 
satisfying that $\phi ( \lambda^+ ) = \frac{1}{k_{\phi}} \lambda^+$, 
$\phi ( \lambda^- ) = k_{\phi} \lambda^-$, 
and that each boundary component of $F_0$ is a cycle of leaves 
of $\lambda^+$ and of $\lambda^-$ and contains singularities of these two foliations. 
See \cite{Thurston} for original definition, 
and see also \cite{FLP} for a survey. 

Now let us take 
a pseudo-Anosov homeomorphism $\phi : F_0 \to F_0$ 
such that the capped off homeomorphism $\widehat{\phi} : F \to F$ 
naturally induced from $\phi$ is isotopic to the identity map. 
Moreover $\widehat{\phi} : F \to F$ can be taken 
to satisfy that 
the point $x_{j+1}$ is mapped to $x_j$ for each $j$ with $1 \le j \le 2b-1$ 
and $x_1$ is mapped to $x_{2b}$. 
The existence of such a homeomorphism is 
essentially guaranteed by the result given in \cite{IchiharaMotegi2005}. 
See also related works \cite{IchiharaMotegi2002,IchiharaMotegi2007,IchiharaMotegi2009}. 

Precisely, for the case of $g \ge 2$, 
we have such a homeomorphism by extending 
\cite[Example 2, p531]{IchiharaMotegi2005}, 
which gives an example for the genus $2$ with $3$ marked points case. 

When $g = 0$ or $1$, we make 
a pseudo-Anosov homeomorphism $\phi : F_0 \to F_0$ as follows. 
Let $F^1_0=F-\mathrm{int} N(x_3 \cup \cdots \cup x_{2b})$, 
which contains $F_0$ as a subsurface. 
Since we are assuming $b \ge 3$ if $g=0$ and $b \ge 2$ if $g=1$, 
the Euler characteristic of $F^1_0$ is negative. 
Then, as in the above, 
by generalizing \cite[Example 1, 2, p531]{IchiharaMotegi2005}, 
we can find a pseudo-Anosov homeomorphism $\phi_1 : F_0 \to F_0$ 
such that $\phi_1$ naturally induces 
a homeomorphism $\widehat{\phi_1} : F^1_0 \to F^1_0$ 
which is isotopic to the identity map and 
interchanges the two points $x_1$ and $x_2$. 
Note that the capped off homeomorphism of $\phi_1$ on $F$ is 
also isotopic to the identity map on $F$. 
We prepare $\phi_2, \cdots, \phi_{b-1}$ in the same way, and 
compose them to obtain $\phi = \phi_{b-1} \circ \cdots \circ \phi_1$. 
Then the pseudo-Anosov homeomorphism $\phi : F_0 \to F_0$ 
satisfies all the conditions we desired above.

\begin{Remark}
If we want to produce just a $b$-component link, not a knot, 
then it suffices to use the result obtained by Kra \cite{Kra}. 
The result given in \cite{IchiharaMotegi2005} is an extension of that in \cite{Kra}. 
Also if we want to make a link of components less than $b$ and at least 2, 
then it suffices to use a natural extension of \cite[Example 1]{IchiharaMotegi2005}. 
\end{Remark}

\subsubsection{Dehn twist}

Next let us take a Dehn twist along a curve $c$ on $F_0$ 
such that the curve $c$ is constructed as follows. 

Prepare diskal neighborhoods $D_j$ 
of the previously chosen arcs $\alpha_j$ ($j=1,\cdots , b$) on $F = F \times \{ 0 \}$, 
and pushing the interior of them into $V_i$ to obtain properly embedded disks 
$D_{i,1} , \cdots , D_{i,b} $ in $V_i$ for $i=1,2$. 
We assume that each $D_{i,j}$ cuts a ball including $\alpha_{i,j}$ from $V_i$. 
Let $P_j$ be the pair of pants which is bounded by 
$\partial D_j$ together with the two components of $\partial F_0$ 
corresponding to $x_{2j-1}$ and $x_{2j}$. 

We here use the result given as \cite[Lemma 5.2]{CampisiRathbun}: 
Given any genus $g$ Heegaard splitting $(V_1,V_2;F)$ of a closed 3-manifold, 
there exist essential disks 
$D'_{i,1} , \cdots , D'_{i,3g-3}$ for $V_i$ for $i=1,2$ 
such that their boundary curves give a pair of pants decompositions on $F$ 
(i.e., their boundary curves decompose $F$ into pairs of pants maximally) 
and that there exists a curve $c_0$ in $F$ 
which \textit{traverses all the seams} of 
all pairs of pants obtained from the both pants decompositions. 
We further assume that 
$D_{i,1} , \cdots , D_{i,b} $ is contained in the same region 
obtained by cutting $V_i$ along $D'_{i,1} , \cdots , D'_{i,3g-3}$ for $i=1,2$. 

Here let us recall the definitions used in the above. 
For a pair of pants $P$, 
a \textit{seam} of $P$ is an essential arc in $P$ 
connecting two distinct boundary components of $P$. 
Let $\mathcal{P}$ be 
a set of pairs of pants embedded into a surface with 
mutually disjoint interiors 
and $\gamma$ a simple closed curve on the surface. 
We say that $\gamma$ \textit{traverses a seam} in $P \in \mathcal{P}$ 
if $\gamma$ intersects $P$ minimally, up to isotopy, and 
a component of $\gamma \cap P$ is a seam. 
We say that $\gamma$ \textit{traverses all the seams} of $\mathcal{P}$ 
if $\gamma$ traverses every seam of every pair of pants in $\mathcal{P}$. 
Please remark that these definitions are slightly different with 
those in \cite{MinskyMoriahSchleimer} and \cite{CampisiRathbun}. 

We further take disks $D''_{i,1}$, $\cdots$, $D''_{i,b-1}$, $D''_{i,b}$ 
in $V_i$ for $i=1,2$ such that, 
together with $D_{i,1} , \cdots , D_{i,b} $ and 
$D'_{i,1} , \cdots , D'_{i,3g-3}$, 
the collection of simple closed curves 
$\{ \partial D_{i,j} \}_{1 \le j \le b} \cup \{ \partial D'_{i,j'} \}_{1 \le j' \le 3g-3} \cup \{ \partial D''_{i,j''} \}_{1 \le j'' \le b}$ divides $F_0$ into 
a set of pairs of pants, say $\mathcal{P}_i$. 
Note that $P_1,\cdots, P_b$ are included in both $\mathcal{P}_i$'s. 
Then we have a simple closed curve $\gamma$ on $F_0$ 
traverses all the seams of pairs of pants in 
$\mathcal{P}'_i  = \mathcal{P}_i - \{ P_1,\cdots, P_b \}$. 
See Figure \ref{FigureDisks} for closed genus two case for example. 

\begin{figure}\begin{center}
 {\unitlength=1cm
 \begin{picture}(11.5,6.2)
\put(0,0){\includegraphics[keepaspectratio]
{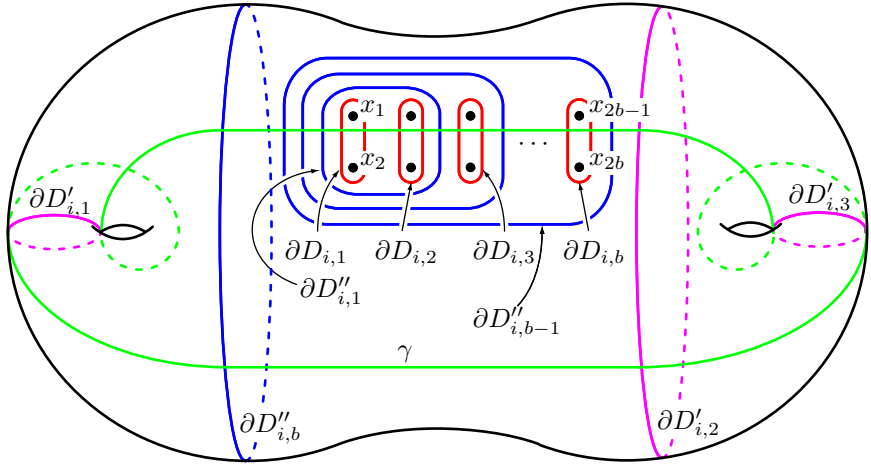}}
 \put(4.7,4.66){$x_1$}
 \put(4.7,3.96){$x_2$}
 \put(7.75,4.66){$x_{2b-1}$}
 \put(7.75,3.96){$x_{2b}$}
 \put(3.7,2.7){$\partial D_{i,1}$}
 \put(4.9,2.7){$\partial D_{i,2}$}
 \put(6.2,2.7){$\partial D_{i,3}$}
 \put(7.4,2.7){$\partial D_{i,b}$}
 \put(3.9,2.2){$\partial D''_{i,1}$}
 \put(6.8,4.15){$\cdots$}
 \put(6.2,1.8){$\partial D''_{i,b-1}$}
 \put(3.1,.4){$\partial D''_{i,b}$}
 \put(5.2,1.4){$\gamma$}
 \put(.3,3.4){$\partial D'_{i,1}$}
 \put(8.65,.45){$\partial D'_{i,2}$}
 \put(10.4,3.45){$\partial D'_{i,3}$}
 \end{picture}}
 \caption{$\partial D''_{i,1} , \cdots , \partial D''_{i,b-1}, \partial
D''_{i,b}$ and $\gamma$ for closed genus two case.}
 \label{FigureDisks}
\end{center}\end{figure}

From the simple closed curve $\gamma$ regarded as on $F$, 
by smoothing $\gamma$ and the arc $\alpha_1$ at the intersection point, 
we have a simple arc on $F$ with endpoints $x_1$ and $x_2$ 
which traverses all the seams of pairs of pants in $\mathcal{P}'_i$ for $i=1,2$. 
We take a diskal neighborhood of the arc on $F$, 
and let $c$ be its boundary curve. 
Then $c$ is also a simple closed curve on $F_0$ 
traverses all the seams of pairs of pants in $\mathcal{P}'_i $ for $i=1,2$. 

Now recall that $\tau$ is a Dehn twist along a curve $c$ on $F_0$. 
This gives a non-trivial map on $F_0$, 
but its natural extension $\widehat{\tau}$ on $F$ 
is isotopic to the identity map on $F$, 
for it bounds a disk on $F$, 
which contains exactly two marked points $x_1$ and $x_2$.

\subsection{Claims}

Here we prove two key claims.

The first claim is a modification 
of \cite[Lemma 2.3]{MinskyMoriahSchleimer}. 
To state the claim, we need to prepare some notation. 
Let $\overline{\mathcal{D} (W_1)}$ be the closure of 
the disk set $\mathcal{D} (W_1) $ of $W_1$ 
in the space of projective measured laminations $\mathcal{PML} (F_0)$ on $F_0$. 
Let $\overline{\mathcal{D} (W_2)}$ be the closure of 
the disk set $\mathcal{D} (W_2) $ of $W_2$, 
which is also regarded as in $\mathcal{PML} ( F_0 )$. 
We here omit the definitions of the terms used here. 
Please refer \cite{FLP} for example. 

\begin{Claim}
With sufficiently large $N>0$, 
the map $\phi_N  = \tau^N \circ \phi \circ \tau^{-N}$ is 
a pseudo-Anosov homeomorphism of $F_0$ 
with stable and unstable laminations $\lambda_N^\pm \in \mathcal{PML} ( F_0 )$ 
such that 
$\lambda_N^+ \not\in \overline{\mathcal{D} (W_1) }$ and 
$\lambda_N^- \not\in \overline{\mathcal{D} (W_2) }$. 
\end{Claim}

\begin{proof}
Since $\phi$ is pseudo-Anosov, 
its conjugate $\phi_N$ is also a pseudo-Anosov map of $F_0$. 
Let $\lambda^\pm$ be the stable/unstable laminations of $\phi$. 
Then the stable/unstable laminations $\lambda_N^\pm$ of $\phi_N$ 
are just $\tau^N ( \lambda^\pm )$ as claimed in \cite{MinskyMoriahSchleimer}. 
Both laminations intersect the curve $c$, since $\lambda^\pm$ is filling. 
Thus, as $N \to \infty$, 
the both laminations $\lambda^\pm_N$ converge to $[c]$ in $\mathcal{PML}(F_0)$, 
which is guaranteed by calculations given in \cite[Appendix C]{FLP}. 

On the other hand, we see $[c] \not\in \overline{\mathcal{D} (W_1)}$ as follows. 
Set $\mathcal{C}_1= \{ \partial D'_{1,j'} \}_{1 \le j' \le 3g-3} \cup \{ \partial D''_{1,j''} \}_{1 \le j'' \le b}$. 
Then $c$ is not isotopic to any element of $\mathcal{C}_1$, 
otherwise $c$ does not traverse 
all the seams of pairs of pants in $\mathcal{P}'_1$, 
contradicting the assumption. 
Suppose that there exists 
a sequence of meridians $\boldsymbol{\mu}=\{ \mu_1 , \mu_2 , \cdots \}$ on $F_0$ 
which gives a sequence $\{ [ \mu_1] , [\mu_2] , \cdots \} \subset \mathcal{D} (W_1)$ 
converging to $[c] \in \mathcal{PML}(F_0)$. 
By passing to a subsequence, 
we may assume that none of the curves in $\boldsymbol{\mu}$ 
is isotopic to a curve in $\mathcal{C}_1$. 
Then we see that every curve in $\boldsymbol{\mu}$ has a \textit{wave} 
(i.e., an essential arc in a pair of pants $P$ 
with endpoints on the same boundary component of $P$) 
in some pair of pants in $\mathcal{P}'_1$, 
since any meridian is either parallel to an element of $\mathcal{C}_1$ 
or has a wave in some pair of pants in $\mathcal{P}'_1$ 
by outermost disk arguments as shown in \cite[Lemma 2.2]{MinskyMoriahSchleimer}. 
Passing to subsequences again, 
we may assume that all the curves in $\boldsymbol{\mu}$ 
have the same wave $w$ in the same pair of pants $Y \in \mathcal{P}'_1$. 
Since $c$ traversed all the seams of pairs of pants in $\mathcal{P}'_1$, 
the curve $\mu$ in $\boldsymbol{\mu}$ with a large enough index 
would also traverse all the seams of pairs of pants in $\mathcal{P}'_1$. 
Since one of the seams of $Y$ intersects $w$ essentially, 
it would follow that the curve $\mu$ 
also has a self-intersection, a contradiction. 
Consequently, for a sufficiently large $N$, 
we see $\lambda^+_N \not\in \overline{\mathcal{D} (W_1)}$. 

By considering $W_2$ and 
by regarding $\mathcal{P}'_2 \subset F\times \{1\} \subset \partial W_2$, 
we can prove $\lambda^-_N \not\in \overline{\mathcal{D} (W_2)}$ in the same way. 
\end{proof}

The next claim is 
a bridge decomposition version of 
\cite[Lemma 2.1]{MinskyMoriahSchleimer}, 
which is essentially proved as \cite[Theorem 2.3]{AbramsSchleimer}, 
and is originally due to \cite{Hempel}. 
We here follow 
the outline of their proof of \cite[Theorem 2.3]{AbramsSchleimer}.

\begin{Claim}
Let $(\phi_N)^m ( \mathcal{D} (W_2) ) $ be 
the set of vertices in $\mathcal{C}(F_0)$ 
each of which corresponds to 
the image of a curve bounding a disk 
in $W_2$ by $(\phi_N)^m$ for $m \in \mathbb{N}$. 
Then 
$d ( \mathcal{D} (W_1) , (\phi_N)^m ( \mathcal{D} (W_2) ) ) \to \infty$ 
as $m \to \infty$.
\end{Claim}

\begin{proof}
Suppose for a contradiction that 
$d( \mathcal{D} (W_1) , (\phi_N)^m ( \mathcal{D} (W_2 ) ) $ 
is bounded by some integer $Z$ for all $m \in \mathbb{N}$. 
Then, for every $m$, 
there is a set of essential curves 
$\{ \alpha_k^m \}_{k=0}^Z$ such that 
$[\alpha_0^m] \in \mathcal{D} (W_1)$, 
$[\alpha_Z^m] \in (\phi_N)^m ( \mathcal{D} (W_2 ) )$, 
and $\alpha_k^m \cap \alpha_{k+1}^m = \emptyset$. 
It follows that there is a curve $\beta_m \in \mathcal{D} (W_2)$ 
such that $(\phi_N)^m ( \beta_m ) = \alpha_Z^m$. 
For these $\beta_m$'s, 
the corresponding points in $\mathcal{PML}(F_0)$ 
avoid an open neighborhood of 
the unstable lamination $\lambda_N^-$, 
since $\lambda_N^-$ is not contained 
in $\overline{\mathcal{D} (W_2)}$ by the above claim. 
Thus, the points $[ (\phi_N)^m ( \beta_m ) ] = [ \alpha_Z^m ]$ 
converge to $\lambda_N^+$ as a sequence in $\mathcal{PML}(F_0)$.
Inductively pass to subsequences exactly $Z$ times 
to ensure that the $k$-th sequence $\{ \alpha^m_k \}$ also gives 
a convergent sequence in $\mathcal{PML}(F_0)$, for $k= Z-1, Z-2,...,0$. 
Denote the limit of the sequence $\{ [ \alpha_k^m ] \}$ by $\lambda_k$. 
Note that $\lambda_Z = \lambda_N^+$, 
in particular, $\lambda_Z \ne \lambda_0$ also by the above claim 
since $\lambda_0 \in \overline{\mathcal{D} (W_1)}$ but 
$\lambda_N^+ \not\in \overline{\mathcal{D} (W_1) }$. 
Then, as $\alpha_k^m \cap \alpha_{k+1}^m = \emptyset$, 
we have $\iota (\lambda_k , \lambda_{k+1}) = 0$, abusing notation slightly, 
where $\iota ( \cdot , \cdot )$ denotes the geometric intersection number. 
Since $\iota$ is a continuous function 
$\mathcal{ML}(F_0) \times \mathcal{ML}(F_0) \to \mathbb{R}$, 
it follows that 
$\lambda_Z$ and $\lambda_{Z-1}$ 
must be the same point of $\mathcal{PML}(F_0)$, 
for $\lambda_N^+$ is minimal and uniquely ergodic. 
Inductively, $\lambda_k = \lambda_{k-1}$ in $\mathcal{PML}(F_0)$ 
which implies that $\lambda_Z = \lambda_0$. This is a contradiction. 
\end{proof}

\subsection{Replacement}

To complete the proof of the Theorem 1, 
we replace $V_2$ in $M$ via 
an extension of the map $\widehat{(\phi_N)^m}$ on $\partial V_2$ as follows. 

Recall that, by constructions, 
the extensions $\widehat{\phi}$ and $\widehat{\tau}$ on $F$ 
are isotopic to the identity on $F$. 
Thus $\phi_N= \tau^N \circ \phi \circ \tau^{-N}$ has 
a natural extension $\widehat{\phi_N}$ on $F$ 
which is isotopic to the identity on $F$. 
Set $m = 2 b m_0 + 1$ for some integer $m_0$, 
and consider the map $(\phi_N)^m$, 
whose natural extension $\widehat{(\phi_N)^m}$ 
is isotopic to the identity on $F$ also. 

Let $\Phi_m : F \times [0,1] \to F \times [0,1]$ be an isotopy 
from $\widehat{(\phi_N)^m}$ to the identity 
with $\Phi_m (x, 0) = ( \widehat{(\phi_N)^m} (x), 0)$ and $\Phi_m (x, 1) = (x, 1)$.
Note that 
$\Phi_m (x_j, 0) = ( \widehat{(\phi_N)^m} ( x_j ), 0) = ( x_{j-1} , 0 )$ 
for $1 \le j \le 2b-1$ and 
$\Phi_m (x_1, 0) = ( \widehat{(\phi_N)^m} ( x_1 ), 0) = ( x_{2b} , 0 )$, 
since $\widehat{\phi}$ maps $x_{j+1}$ to $x_j$ and $x_1$ to $x_{2b}$, 
$\widehat{\tau}$ keeps $x_j$'s invariant, 
and $m = 2 b m_0 + 1$ with $m_0 \in \mathbb{Z}$. 
Regarding $F \times [0,1] \subset V_2$, 
this map $\Phi_m : F \times [0,1] \to F \times [0,1]$ 
can be naturally extended 
to the map $\widehat{\Phi_m} : V_2 \to V_2$, 
since $\Phi_m |_{F \times \{ 1 \}}$ is the identity map. 

Now we replace $V_2$ by $\widehat{\Phi_m} (V_2)$. 
Precisely, when we regard $M$ as $V_1 \cup_h V_2$ 
with a gluing map $h : \partial V_1 \to \partial V_2$, 
we set $M_m$ as $V_1 \cup_h \widehat{\Phi_m} (V_2)$. 

Note that this manifold $M_m$ can be regarded as 
$V_1 \cup_{(\widehat{\Phi_m})^{-1} \circ h}  V_2$ 
with the gluing map $(\widehat{\Phi_m})^{-1} \circ h : \partial V_1 \to \partial V_2$. 
Since $(\widehat{\Phi_m})^{-1}$ is isotopic to the identity map on $\partial V_2$, 
the gluing map $(\widehat{\Phi_m})^{-1} \circ h$ is isotopic to $h$. 
This impiles that $M_m$ is homeomorphic to $M$. 
In particular, $(\widehat{\Phi_m})^{-1} \circ h (\partial V_1) = \partial V_2$ 
in $M_m$ is identified with $F$ in $M$. 

In $M_m$, we have a knot $K_m$, 
as a union of 
$\alpha_{1,1} \cup  \cdots \cup \alpha_{1,b}$ in $V_1$ and 
$\widehat{\Phi_m} (\alpha_{2,1} \cup  \cdots \cup \alpha_{2,b})$ in $\widehat{\Phi_m} (V_2)$. 
This $K_m$ has a $b$-bridge presentation with respect to 
the Heegaard surface obtained from $\partial V_i$'s. 
Under the identification observed as above, 
the knot $K_m$ can be regarded as in $M$, and 
it has a $b$-bridge presentation with respect to $F$. 

Note that the distance of the bridge splitting of $K_m$ with respect to $F$ 
is equal to 
$d ( \mathcal{D} (W_1) , \mathcal{D} ( \widehat{\Phi_m} (W_2) ) $ 
in $\mathcal{C}(F_0)$. 
On the other hand, by construction, 
we see that $\widehat{\Phi_m} |_{F_0} = (\phi_N)^m$. 
It follows that 
$\mathcal{D} ( \widehat{\Phi_m} (W_2) ) = 
(\phi_N)^m ( \mathcal{D} (W_2) ) )$ in $\mathcal{C}(F_0)$. 
Thus the distance of the bridge splitting of $K_m$ 
is equal to 
$d ( \mathcal{D} (W_1) , (\phi_N)^m ( \mathcal{D} (W_2) ) )$ 
in $\mathcal{C}(F_0)$. 

Consequently, together with Claim 2, we see that 
the distance of the bridge splitting of $K_m$ goes to infinity as $m \to \infty$. 
This completes the proof of Theorem \ref{Thm}.

\section{Application}

Let $K$ be a knot in a closed $3$-manifold $M$. 
The \textit{meridional destabilizing number} of $K$, 
denoted by $\mathit{md}(K)$, is defined by the maximal number of $m$ 
such that $(M,K)$ admits a $(t(K)-m+1,m)$-bridge splitting. 
Here $t(K)$ is the tunnel number of $K$, i.e., 
the minimal value $t$ such that $(M,K)$ admits a Heegaard surface of genus $t+1$. 
See \cite{Saito2} for the original definition. 
Note that $\mathit{md}(K)\le t(K)+1$. As an application of Theorem \ref{Thm}, 
we prove Corollary \ref{cor} described in Section 1 
by using the following Tomova's theorem. 

\begin{Lemma}[c.f. {\cite[Theorem 10.3]{Tomova}}]\label{Lem}
Let $K$ be a non-trivial knot in a closed orientable irreducible 3-manifold $M$ 
and $\mathcal{S}=(S,S\cap K)$ a bridge surface of $(M,K)$ 
such that $|S\cap K|\le 4$ if $S$ is a $2$-sphere. 
If $\mathcal{S}'=(S',S'\cap K)$ is also a bridge surface 
or a Heegaard surface of $(M,K)$, 
then 
\begin{enumerate}
\item $d(\mathcal{S})\le 2-\chi(S'-K)$ or 
\item $\mathcal{S}'$ is obtained from $\mathcal{S}$ by stabilization, 
meridional stabilization and perturbation. 
\end{enumerate}
\end{Lemma}

\begin{proof}[Proof of Corollary \ref{cor}]
Let $t, m$ be a pair of integers satisfying $t \ge 1$ and $m \ge 0$ with $m \le t+1$. 
We will show that there exists a knot $K\subset S^3$ 
of tunnel number $t$ and of meridional destabilizing number $m$.

First consider the case of $m=0$. 
It follows from \cite[Theorem 4.2]{MinskyMoriahSchleimer} that 
for any positive integer $t\ge 1$, there is a knot $K\subset S^3$ 
with $t(K)=t$ so that $K$ admits no $(t,1)$-bridge surface. 
This implies that there exists a knot $K\subset S^3$ 
with $t(K)=t$ and $\mathit{md}(K)=0$. 

Next we consider $m>0$. 
It follows from Theorem \ref{Thm} that 
there exists a knot $K\subset S^3$ so that $K$ admits 
a $(t-m+1,m)$-bridge surface $\mathcal{S}$ of distance greater than $2t+2$. 
Then the knot $K$ is shown to satisfy $t(K)=t$ and $\mathit{md}(K)=m$ as follows. 

We have a $(t+1,0)$-bridge surface from $\mathcal{S}$ by $m$-times 
meridional stabilizations, and hence we see $t(K)\le t$. 
Suppose that $t(K)<t$, i.e., $K$ admits a $(t,0)$-bridge surface $\mathcal{S}'$. 
Then, since $d(\mathcal{S}) > 2t = 2-\chi(S'-K)$, 
it follows from Lemma \ref{Lem} that $\mathcal{S}'$ is obtained from $\mathcal{S}$ 
by stabilization, meridional stabilization and perturbation. 
We can observe that 
a $(t-m+2,m)$-bridge surface is obtained 
from $\mathcal{S}$ by one-time stabilization, 
a $(t-m+2,m-1)$-bridge surface by one-time meridional stabilization, and 
a $(t-m+1,m+1)$-bridge surface by one-time perturbation. 
These imply that we cannot have a $(t,0)$-bridge surface, which contradicts 
that $\mathcal{S}'$ is obtained from $\mathcal{S}$ by such operations. 
Hence we see $t(K)=t$. 
Similarly, we see that $K$ admits no $(t-m,m+1)$-bridge surface 
and hence $md(K)\le m$. 
This implies $\mathit{md}(K)=m$. 
\end{proof}

\begin{Remark}
We can improve Corollary \ref{cor} as follows: 
For any integers $t \ge 1$, $m \ge 0$ with $m \le t+1$ and $b\ge 0$, 
there exists a knot $K\subset S^3$ with $t(K)=t$ and $\mathit{md}(K)=m$ 
which admits no $(t-m,b)$-bridge splitting. This is pointed out by Yo'av Rieck. 
\end{Remark}

\bibliographystyle{amsplain}

\end{document}